\documentclass[reqno, 12pt]{amsart}
\usepackage[letterpaper,hmargin=1in,vmargin=1in]{geometry}
\usepackage{amsmath,amssymb,amsthm, hyperref}
\usepackage{color}

\newtheorem{theorem}[equation]{Theorem}

\newtheorem{corollary}[equation]{Corollary}
\newtheorem{lemma}[equation]{Lemma}

\theoremstyle{remark}
\newtheorem{remark}[equation]{Remark}
\newtheorem{question}[equation]{Question}
\newtheorem{example}[equation]{Example}

\numberwithin{equation}{section}

\newcommand{\oo}{\mathcal{O}}
\newcommand{\Vol}{\operatorname{Vol}}

\newcommand{\D}{\mathcal{D}}

\begin{document}
\title[The Steklov spectrum of surfaces: asymptotics and invariants]
{The Steklov spectrum of surfaces: asymptotics and invariants}
\author{Alexandre Girouard}
\address{D\'e\-par\-te\-ment de math\'ematiques et de
sta\-tistique, Universit\'e Laval, Pavillon Alexandre-Vachon,
1045, av. de la M\'edecine,
Qu\'ebec Qc G1V 0A6, 
Canada }
\email{alexandre.girouard@mat.ulaval.ca}
\author{Leonid Parnovski}
\address{Department of Mathematics, University College London, 
Gower Street, London WC1E 6BT, UK}
\email{leonid@math.ucl.ac.uk}
\author{Iosif Polterovich}
\address{D\'e\-par\-te\-ment de math\'ematiques et de
sta\-tistique, Univer\-sit\'e de Mont\-r\'eal CP 6128 succ
Centre-Ville, Mont\-r\'eal QC  H3C 3J7, Canada.}
\email{iossif@dms.umontreal.ca}
\author{David A. Sher}
\address{Department of Mathematics, University of Michigan, 2074 East Hall, 530 Church Street, Ann Arbor, MI  48109-1043, U.S.A.}
\email{dsher@umich.edu}

\begin{abstract}
  We obtain precise asymptotics for the Steklov
  eigenvalues on a compact Riemannian surface with boundary.  It is
  shown that the number of connected components of the boundary, as
  well as their lengths, are invariants of the Steklov spectrum. The
  proofs are based on pseudodifferential techniques for the
  Dirichlet-to-Neumann operator and on a number--theoretic argument. 
\end{abstract}
\maketitle

\section{Introduction and main results}
\subsection{Steklov spectrum}
Let $\Omega$ be a smooth compact Riemannian manifold of dimension $n$
with smooth boundary $M=\partial \Omega$ of dimension $n-1$. Consider
the Steklov eigenvalue problem on $\Omega$:
\begin{equation}
\label{stek}
\begin{cases}
  \Delta u=0& \mbox{ in } \Omega,\\
  \frac{\partial u}{\partial\nu}=\sigma \,u& \mbox{ on
  }M.
\end{cases}
\end{equation}
Its spectrum is discrete and is given by a sequence of eigenvalues
$$0=\sigma_0\leq \sigma_1 \le \sigma_2 \le \dots \nearrow \infty.$$
The Steklov eigenvalues are the eigenvalues
of the Dirichlet-to-Neumann operator\linebreak
$\D:C^{\infty}(M) \rightarrow C^\infty(M)$; recall that if $f\in C^{\infty}(M)$, then $\D f=\partial_\nu(Hf)$, where $Hf \in C^\infty(\Omega)$ is the harmonic extension of $f$ to
$\Omega$ and $\partial_\nu$ denotes the outward normal derivative. 

\subsection{Spectral asymptotics}
The Dirichlet-to-Neumann operator $\D$ is an elliptic self-adjoint pseudodifferential operator of order one
(see \cite[pp. 37-38]{Ta}). Its eigenvalues satisfy the asymptotic formula  
\begin{equation}
\label{Weylaw}
\sigma_j= 2\pi \left(\frac{j}{\Vol(\mathbb{B}^{n-1})\, \Vol(M)}\right)^\frac{1}{n-1} + O(1),
\end{equation}
which is a direct consequence of  Weyl's law with a sharp remainder estimate (see \cite{ho}):
$$
\#(\sigma_j < \sigma)=\frac{\Vol(\mathbb{B}^{n-1})\,\Vol (M)}{(2\pi)^{n-1}} \sigma^{n-1} + O(\sigma^{n-2}).
$$
Here $\mathbb{B}^{n-1}$ is a unit ball in $\mathbb{R}^{n-1}$. For simply connected surfaces, much more precise asymptotics
were independently obtained by Rozenblyum and Guillemin--Melrose
(see \cite{Ro, e}):
\begin{equation}
\label{refined}
\sigma_{2j}=\sigma_{2j+1} + \mathcal O(j^{-\infty})=\frac{2\pi}{\ell(M)}j+\mathcal O(j^{-\infty}) .
\end{equation}
Here $\ell(M)$ denotes the length of the boundary $M$, and the notation $\mathcal O(j^{-\infty})$ means that the error term decays faster than any power of $j$.


The first goal of this paper is to prove an analogue of \eqref{refined} for an {\it arbitrary} surface $\Omega$. Given a finite sequence $C=\{\alpha_1,\cdots,\alpha_k\}$ of positive numbers,
consider the following union of multisets (i.e. sets with multiplicity):
$\{0,..\dots,0\}\cup \alpha_1 \mathbb{N}\cup \alpha_1 \mathbb{N}\cup
\alpha_2\mathbb N\cup\alpha_2\mathbb N\cup \dots
\cup\alpha_k\mathbb{N}\cup \alpha_k\mathbb{N}$, where the first multiset
contains $k$ zeros and $\alpha\mathbb{N}=\{\alpha,2\alpha,3\alpha,\dots,n\alpha,\dots\}$.
We rearrange the elements of this multiset into a monotone increasing sequence $S(C)$.
For example,
$S(\{1\})=\{0,1,1,2,2,3,3,\cdots\}$ and $S(\{1, \pi\})=\{0,0,1,1,2,2,3,3,\pi,\pi,4,4,5,5,6,6, 2\pi,2\pi,\\7,7,\cdots\}$.
\begin{theorem}\label{Theorem:MainAsymptotics}
  Let $\Omega$ be a smooth compact Riemannian surface with boundary $M$.
  Let $M_1,\cdots,M_k$ be the connected components of the
  boundary $M=\partial\Omega$, with lengths $\ell(M_i), 1\leq i\leq
  k$. 
  Set $R=\left\{\frac{2\pi}{\ell(M_1)},\cdots,\frac{2\pi}{\ell(M_k)}\right\}$. Then
  $$\sigma_{j}=S(R)_j+\oo(j^{-\infty}).$$
\end{theorem}
\noindent Theorem \ref{Theorem:MainAsymptotics} is proved in Section \ref{sec:spectral}.

Note that the sequence $S(\{ \alpha\})$ is the Steklov spectrum of a disk of radius $1/\alpha$; as a consequence, the sequence $S(R)$ is the Steklov spectrum of a disjoint
    union of $k$ disks, with radii
    $\frac{1}{2\pi}\ell(M_i)$, $i=1,\cdots,k$.

\begin{remark}
  Some related results  in higher dimensions were obtained in
  \cite{HL} (see also Theorem \ref{higher});  this  paper served as a
  motivation and a starting point for our research.
\end{remark}

Theorem \ref{Theorem:MainAsymptotics} has the following corollary, which generalizes \cite[Proposition 1.5.2]{KKP}:
\begin{corollary} For any smooth compact Riemannian surface $\Omega$ with $k$ boundary components, there is a constant $N$ depending on the metric on $\Omega$ such that for $j>N$, the multiplicity of $\sigma_j$ is at most $2k$.
\end{corollary}
The proof is immediate from Theorem \ref{Theorem:MainAsymptotics}.

\subsection{Spectral invariants} It follows from the standard results
of Duistermaat and Guillemin \cite{DG}  on wave trace asymptotics for
pseudodifferential operators that, in any dimension, the lengths of
closed geodesics on the boundary $M$ are invariants of the spectrum of
the Dirichlet-to-Neumann operator. For surfaces, the boundary is
one--dimensional, and the lengths of closed geodesics are simply the
integer multiples of the lengths of boundary components. However, this
information is in general not enough to determine the number of
boundary components and their lengths. For example, the wave traces of
a surface $\Omega_1$ with two boundary components of lengths $1$ and
$5$, respectively, and a surface $\Omega_2$ with three boundary
components of lengths $1$, $2$ and $3$, respectively, have the same
singularities. The perimeters of $\Omega_1$ and $\Omega_2$ are also
equal. Hence, these surfaces can not be spectrally distinguished
without  further analysis.

Instead of using the wave trace, we propose a more direct approach to the study of spectral invariants,  based on Theorem \ref{Theorem:MainAsymptotics}
and some elementary number theory. The main result of the paper is:
\begin{theorem}\label{theorem:invspec}  
The Steklov spectrum determines the number  and the lengths of  boundary components of a smooth compact Riemannian surface.
\end{theorem}
 The proof of Theorem \ref{theorem:invspec} based on Lemmas \ref{limsup} and \ref{bijection} is presented in section \ref{proofmain}.  In particular, this theorem implies that the disk is uniquely determined by its Steklov spectrum among all  bounded  smooth Euclidean domains; see Section~\ref{appl}.
Let us also state the following corollary of Lemma \ref{limsup}:
\begin{corollary}
\label{maxelem} Let $\{\sigma_j\}$ be the monotone increasing sequence of Steklov eigenvalues of a smooth compact Riemannian surface $\Omega$. Then the length $\ell_{max}$ 
of a boundary component of $\Omega$ with the largest perimeter is given by:
  $$\ell_{max}=\frac{2\pi}{\limsup_{j\rightarrow\infty}(\sigma_{j+1}-\sigma_j)}.$$
  \end{corollary}

Interestingly enough, Theorem \ref{theorem:invspec}  does not admit a straightforward generalization to higher dimensions, as the following example shows.
\begin{example}
\label{higherdim}
Consider four flat rectangular tori: $T_{1,1}=\mathbb{R}^2/\mathbb{Z}^2$,  $T_{2,1} =\mathbb{R}/2\mathbb{Z}\times\mathbb{R}/\mathbb{Z}$,
$T_{2,2}~=~\mathbb{R}^2/(2\mathbb{Z})^2$ and $T_{\sqrt{2},\sqrt{2}}= \mathbb{R}^2/(\sqrt{2}\mathbb{Z})^2$. It was shown in (\cite{DR, P})
that the disjoint union $\mathcal{T}=T_{1,1}\sqcup T_{1,1}\sqcup T_{2,2}$ is Laplace--Beltrami isospectral to the disjoint union $\mathcal{T}'=T_{2,1}\sqcup T_{2,1}\sqcup 
T_{\sqrt{2},\sqrt{2}}$. Therefore, for any $L>0$, the two disjoint unions of cylinders $\Omega_1=[0,L]\times \mathcal{T}$  and $\Omega_2=
[0,L] \times \mathcal{T}'$ are  Steklov isospectral. This follows from separation of variables  (see \cite[Lemma 6.1]{CEG}). At the same time, 
$\Omega_1$ has four boundary components of area $1$ and two  boundary components of area $4$, while $\Omega_2$ has  six 
boundary components of  area $2$.  Therefore, the areas of boundary components can not be determined from the Steklov spectrum.
\end{example}
\begin{remark}
As was mentioned above, the error estimate \eqref{Weylaw}  in dimension $n\ge 3$  is significantly weaker than \eqref{refined}.
It does not allow one to ``decouple'' the contributions of different
boundary components to the whole spectrum,  and this explains why the volumes of individual boundary components can not be recovered from the Steklov spectrum in higher dimensions.  
An attempt to prove a result of this kind was made in \cite[Theorem 1.5]{HL}; however, the proof lacked a ``decoupling'' argument. 
As Example \ref{higherdim} shows, such an argument does not exist in
dimensions $n\ge 3$. It remains to be seen whether the number of
boundary components can be determined from the Steklov spectrum in
higher dimensions.
\end{remark}
\subsection{Discussion} 
\label{appl}
One could ask whether there exist Riemannian manifolds with boundary
which are not isometric but have the same Steklov spectrum; in fact, there are at least two general constructions of such manifolds. 
The first one is based on the idea used in Example~\ref{higherdim}. Namely, cylinders of the same length over Laplace--Beltrami isospectral closed manifolds (of which there exist many examples --- see, for instance, \cite{Gor} and references therein) are Steklov isospectral.   Using this method, one can produce examples of non-isometric 
Steklov isospectral manifolds of any dimension $n\ge 3$.

In dimension two, one can use a different approach. Let  $g_1$ and
$g_2= \rho g_1$ be two conformally equivalent metrics on a surface
$\Omega$ with the conformal factor $\rho|_{\partial \Omega} \equiv
1$. Then it immediately follows from the variational principle for
Steklov eigenvalues and from the conformal invariance of the Dirichlet
energy that the Riemannian surfaces $(\Omega, g_1)$ and $(\Omega, g_2)$ are isospectral. In \cite{FS} such surfaces are referred to as {\it $\sigma$-isometric}. It is conjectured
in \cite{JS} that two Riemannian surfaces are Steklov isospectral if and only if they are $\sigma$--isometric. If true, this would imply that any smooth planar domain is uniquely determined by its Steklov spectrum, since two planar domains are
$\sigma$--isometric if and only if they are isometric.

Note that in both constructions presented above, the Steklov isospectral manifolds have Laplace isospectral boundaries. Therefore, it is natural to ask the following:
\begin{question}
\label{quest}
Do there exist Steklov isospectral manifolds with boundaries which are not Laplace isospectral?
\end{question}

In dimension $n\ge 3$ this question remains open.  At the same time, Theorem  \ref{theorem:invspec} implies that  in two dimensions the answer is negative:

\begin{corollary} Let $\Omega_1$ and $\Omega_2$ be two Steklov isospectral Riemannian surfaces. Then $\partial \Omega_1$ an $\partial \Omega_2$ are Laplace isospectral.
\end{corollary}
\begin{proof}
Indeed, two closed curves have the same Laplace spectrum if and only if they have the same length. The corollary then follows immediately from Theorem \ref{theorem:invspec}.
\end{proof}

Another interesting problem is to determine which Riemannian manifolds are uniquely determined by their Steklov spectrum. In \cite{PS} it is conjectured that the $n$-dimensional ball is uniquely determined by its Steklov spectrum
among all domains in $\mathbb{R}^n$.  
It is proved for $n=2$ \cite{We, e}  and $n=3$ \cite{PS} in the class of smooth Euclidean domains with connected boundary. 
Theorem \ref{theorem:invspec} allows us to remove the assumption that the boundary is connected in dimension two. In fact, we obtain a more general result.
\begin{corollary}
\label{cor}
Let $\Omega$ be a smooth orientable surface of genus zero which is Steklov isospectral to a disk of perimeter $l$. Then $\Omega$ is $\sigma$--isometric to a disk of perimeter $l$.
\end{corollary}
\begin{proof} Indeed, as follows from Theorem \ref{theorem:invspec},
  $\Omega$ has one boundary component of length $l$.  As was shown in
  \cite[Formula (4.6)] {We} (see also \cite[Section 4]{FS2}), for
  orientable surfaces of genus zero with one boundary component of
  length $l$, $\sigma_1$  attains its maximum if and only if the
  surface is $\sigma$--isometric to a disk. This completes the proof
  of the corollary.
\end{proof}
Using  the results of \cite{FS}, similar rigidity statements can be
proved for the critical catenoid and the critical M\"obius band (see
\cite{FS, FS2} for the definitions of these surfaces).
\begin{remark}
One can also show that the disk is uniquely determined by its Steklov
spectrum among all simply connected planar domains with $C^1$
boundaries. Indeed, the one-term Weyl asymptotics hold in this case
\cite{Agr}, and Weinstock's inequality is true under even more general
assumptions \cite{GP}.  It would be interesting to prove Theorem
\ref{theorem:invspec} for non-$C^\infty$ boundaries.  This requires new
methods, since the corresponding Dirichlet-to-Neumann operator would
no longer be pseudodifferential in this case.
\end{remark}

Finally, one may also ask if orientability is an invariant of the  Steklov spectrum --- for instance,  whether one can always  distinguish between a M\"obius band and a topological disk of the same perimeter. 
Both surfaces have one boundary component, and therefore Theorem  \ref{theorem:invspec} is not sufficient to tell them apart.

\subsection*{Acknowledgments} The authors would like to thank
A. Granville, P. Hislop, and L.~Pol\-tero\-vich for useful discussions.
The research of A.G. was partially supported by a  startup grant
from Universit\'e Laval. The research of L.P. was partially supported
by the EPSRC grant EP/J016829/1. The research of I.P. was partially
supported by NSERC, FRQNT, and the Canada Research Chairs program. The
research of D.S. was partially  supported by the CRM-ISM postdoctoral
fellowship and the NSF grant 1045119.

\section{Proofs}
\subsection{Smoothing perturbations}
\label{sec:prelim}
Recall that a pseudodifferential operator $S$ on a Riemnnian manifold $M$ is called \emph{smoothing} if it has a smooth integral kernel \cite{Treves}; it is a standard fact that smoothing operators form an ideal in the algebra of pseudodifferential operators. Additionally, smoothing operators are bounded as maps from $H^s(M)$ to $H^t(M)$ for any $s$ and $t\in\mathbb R$, so in particular are bounded operators on $L^2(M)$. We will make use of the following well-known result, which is part of
the folklore of the theory of pseudodifferential operators. 
\begin{lemma}\label{lemma:smoothingdiff}
  Let $M$ be a compact manifold of dimension $n$. Let $P$ and $Q$ be elliptic,
  bounded below, self-adjoint pseudodifferential operators on $M$ of
  order $m>0$. If  the difference $P-Q$ is a smoothing operator, then
  the eigenvalues of $P$ and $Q$ satisfy
  $$\lambda_j(P)-\lambda_j(Q)=\oo(j^{-\infty}).$$
\end{lemma}
Since we could not locate a proof in the literature with the desired
level of generality, we have included one here for
convenience. The argument  is based loosely on the proof of the first theorem in \cite[Section 5]{HL}; see also~\cite{e} for a related result.

\begin{proof}[Proof of Lemma~\ref{lemma:smoothingdiff}]
  It follows from the spectral theorem (see
  \cite[Theorem 8.3, p. 71]{sh}) that both $P$ and $Q$ have
  bounded below discrete spectra, which consist entirely of
  eigenvalues, each with finite multiplicity, and there are
  corresponding complete orthonormal bases of $L^2(M)$. 
  Suppose that  the functions $\phi_j$ form an orthonormal basis associated to
  the eigenvalues $\lambda_j(P)$, with $j\geq 1$. Let $E_k\subset
  L^2(M)$ be the span of the first $k$ eigenfunctions
  $\phi_1,\cdots,\phi_k$.
  By hypothesis, the operator $S=Q-P$ is smoothing.
  It follows from the variational characterizations of eigenvalues that
  \begin{align*}
    \lambda_{k+1}(Q)=\lambda_{k+1}(P+S)
    \geq \min_{f\perp E_k,\|f\|=1}
    \left(
      \langle
      Pf,f\rangle+\langle Sf,f\rangle
    \right)
    \geq \lambda_{k+1}(P)-\max_{f\perp E_k,\|f\|=1}
    |\langle Sf,f\rangle|.
  \end{align*}
  In other words,
  \begin{gather}\label{Inequality:SmoothingDifferenceEigenvalues}
    \lambda_{k+1}(P)-\lambda_{k+1}(Q)\leq
    \max_{f\perp E_k,\|f\|=1}|\langle Sf,f\rangle|\leq\max_{f\perp E_k,\|f\|=1}||Sf||_{L^2}.
  \end{gather}
 We may write any such $f$ as a Fourier series: $f=\sum_{j>k}f_j\phi_j$, where $f_j=\langle f,\phi_j\rangle$. Then for any positive integer $p$, \[||Sf||_{L^2}\leq \sum_{j>k}|f_j|\cdot ||S\phi_j||_{L^2}=\sum_{j>k}|f_j|\lambda_j^{-p}||SP^p\phi_j||_{L^2}\leq\lambda_{k}^{-p}\sum_{j>k}|f_j|\cdot ||SP^p\phi_j||_{L^2}.\]
 By Cauchy-Schwarz and Plancherel's theorem, and the fact that $f$ has norm 1,
 \begin{equation}||Sf||_{L^2}\leq\lambda_k^{-p}(\sum_{j>k}||SP^p\phi_j||_{L^2}^2)^{1/2}\leq\lambda_k^{-p}(\sum_{j=1}^{\infty}||SP^p\phi_j||_{L^2}^2)^{1/2}. \end{equation}
 Note that the bound is independent of $f$. Now observe that for any $N$, $||SP^p\phi_j||_{L^2}^2=\lambda_j^{-2N}||SP^{p+N}\phi_j||_{L^2}^2$. Since $S$ is smoothing, so is $SP^{p+n}$ (by the ideal property of smoothing operators). As a consequence, the operator $SP^{p+N}$ is bounded from $L^2$ to $L^2$ for any $N$, and hence $||SP^{p+N}\phi_j||_{L^2}$ is bounded by a constant $K_{p,N}$ independent of $j$. We have for any $p$ and $N$, and any $f$ as above:
 \[||Sf||_{L^2}\leq\lambda_k^{-p}K_{p,N}(\sum_{j=1}^{\infty}\lambda_{j}^{-2N})^{1/2}.\]
 By the Weyl asymptotics, the sum on the right-hand-side is finite for sufficiently large $N$. Fix such a large $N$; then there is a constant $C_p$ depending only on $p$ (and not on $f$) such that $||Sf||_{L^2}\leq C_p\lambda_k^{-p}.$ Using the Weyl asymptotics again, we obtain that for any $p$, there is a $C_p$ such that
\[ \lambda_{k+1}(P)-\lambda_{k+1}(Q)\leq\max_{f\perp E_k,\|f\|=1}||Sf||_{L^2}\leq C_pk^{-mp/n}.\]
Since $m>0$ and the roles of $P$ and $Q$ are symmetric in this argument, this completes the proof. 
\end{proof}

\subsection{Proof of Theorem~\ref{Theorem:MainAsymptotics}}\label{sec:spectral}
  For each $i=1,\cdots,k$, let $\Omega_i$ be a
  topological disk with a Riemannian metric which is isometric to $\Omega$ in a neighborhood of
  the boundary component $M_i$. Let $\Omega_\sharp$ be the disjoint
  union of the disks $\Omega_i$. In other words, $\Omega_\sharp$ is obtained by
  keeping a collar neighborhood of each boundary curve $M_i$ and
  capping it by smoothly gluing a disk. 
  Since $\Omega$ and $\Omega_\sharp$ are isometric in a neighborhood of their
  common boundary $M$, it follows from~\cite[Section 1]{lu} that
  the Dirichlet--to--Neumann operators  $\D_\Omega,
  \D_{\Omega_\sharp}\in OPS^1(M)$ have the same full symbol. In other
  words, the difference $\D_\Omega-\D_{\Omega_\sharp}$ is a smoothing operator.
  From Lemma~\ref{lemma:smoothingdiff}, it follows that
  $$\sigma_j(\Omega)-\sigma_j(\Omega_\sharp)=\oo(j^{-\infty}).$$

  It follows from the Riemann mapping theorem that for each
  $i=1,\cdots,k$, the disk $\Omega_i$ is
  conformally equivalent to the unit disk $\mathbb{D}$. Since
  $\partial\Omega_i$ is smooth, this implies that
  $\overline{\Omega_i}$ is isometric to $(\overline{\mathbb{D}},\delta_i^2g_0)$,
  where $g_0$ is the Euclidean metric and $\delta_i\in
  C^\infty(\overline{\mathbb{D}})$ is a smooth positive function on the closure
  of $\mathbb{D}$.
  In the coordinates provided by this isometry, the Steklov eigenvalue
  problem becomes
  $$\Delta u=0,\quad \partial_\nu u=|\delta_i|\sigma u.$$
  Since $\int_0^{2\pi}|\delta_i|=l_i$ is the length of the boundary
  component $M_i$, it follows from~\cite[Corollary 1]{Ro} that 
  $$\sigma_j(\mathbb{D},\delta_i)=S\left(\left\{\frac{2\pi}{\ell(M_i)}\right\}\right)_j+\oo(j^{-\infty}).$$
  This completes the proof of Theorem~\ref{Theorem:MainAsymptotics}.

\begin{remark}
  The result in~\cite{Ro} is based on a local coordinate
  computation of the full symbol of the Dirichlet--to--Neumann map
  using a graph parametrization of the boundary. For a 
  different approach using a conformal equivalence to the upper
  half-plane, see~\cite{e}.
\end{remark}

The first part of the proof of Theorem \ref{Theorem:MainAsymptotics} admits a straightforward extension to higher dimensions, yielding the following theorem which can be viewed 
as the main result of \cite{HL}:
\begin{theorem} Suppose that $\Omega_1$ and $\Omega_2$ are smooth compact Riemannian manifolds with boundary $M_1$ and $M_2$ respectively; let their Steklov eigenvalues be $\{\sigma_j(\Omega_1)\}$ and $\{\sigma_j(\Omega_2)\}$ respectively. Assume there exists an isometry $\phi$ between a neighborhood $U_1$ of $M_1$ in $\Omega_1$ and a neighborhood $U_2$ of $M_2$ in $\Omega_2$ for which $\phi(M_1)=M_2$. Then 
\[\sigma_j(\Omega_1)-\sigma_j(\Omega_2)=\mathcal O(j^{-\infty}).\]
\label{higher}
\end{theorem}

\subsection{Proof of Theorem \ref{theorem:invspec}}\label{proofmain}
Let  $\sigma=\{\sigma_j\}$ be the monotone increasing sequence of Steklov eigenvalues of a surface with boundary $\Omega$. 
Let $R=\left\{\frac{2\pi}{\ell_1},\cdots,\frac{2\pi}{\ell_k}\right\}$ be a finite multiset, such that $\ell_1, \dots \ell_k$ are
the lengths of the boundary components of $\Omega$. It follows from Theorem~\ref{Theorem:MainAsymptotics} that $\sigma_j-S(R)_j=\oo(j^{-\infty})$. 
In order to prove Theorem \ref{theorem:invspec} we describe an inductive  procedure which allows one to
determine the multiset $R$ from the infinite sequence $\sigma$.  
A problem of independent interest (which we do not address here) is to find a 
practical implementation of the proposed algorithm, that could be used to find the lengths of the boundary components of a surface from its Steklov spectrum
with high precision.

In what follows, we deal only with countable multisets of non-negative real numbers.  Somewhat abusing notation, we identify such a multiset with the 
monotone increasing sequence of its elements. In particular, mappings between multisets are understood as mappings between the corresponding sequences.

Given two multisets of positive real numbers $A$ and $B$, we say that a  
mapping $F:A\rightarrow B$ is \emph{close} if it has the property that
for each $\epsilon>0$, there are only finitely many $x\in A$ with
$|F(x)-x|\geq\epsilon$. We say that $F$ is an \emph{almost-bijection}
if for all but finitely many $y\in B$ the pre-image $F^{-1}(y)$
consists of one point.

\begin{lemma}\label{limsup} 
  Let $A=\{A_j\}$ be a multiset for which there exists a close
  almost-bijection $F:S(R)\rightarrow A$ for some finite multiset $R$
  of positive numbers. Then the smallest element of $R$ is
  $L=\limsup_{j\rightarrow\infty}(A_{j+1}-A_j)$.
\end{lemma}

Now assume without loss of generality that $L=1$ (otherwise we may
divide all elements by $L$). Let $\hat R$ be $R$ with a $1$
removed. We construct a new multiset $\hat A$ as follows. Let $N_0 $ be a
(large) number such that for any natural number $j>N_0$ there exists at least two elements of $A$ at a distance
less than $1/10$ from $j$.  Denote by $G_1(j)$ and $G_2(j)$
the two elements of $A$ which are closest to $j$ (in case of ties, we start
by choosing the largest).
Set
\begin{equation}
\label{hat}
\hat A:=A\setminus\cup_{j>N_0}\{G_1(j),G_2(j)\}.
\end{equation}
We claim
\begin{lemma}\label{bijection} If $\hat A$ is infinite, there is a close almost-bijection $\hat F:S(\hat R)\rightarrow \hat A$.
\end{lemma}
Assuming these two lemmas we may argue as follows.
Note that the monotone increasing sequence $\sigma=\{\sigma_j\}$ of Steklov eigenvalues, viewed as a multiset, satisfies the
hypothesis of  Lemma  \ref{limsup}.  Indeed,  by Theorem \ref{Theorem:MainAsymptotics}, the map from $S(R)$ to $\sigma$
which takes $S(R)_j$ to $\sigma_j$ is a close bijection between sequences. 
Therefore, 
$$\frac{2\pi}{\ell_{max}}=\limsup_{j \to \infty} (\sigma_{j+1}-\sigma_j),$$
where $\ell_{max}=\max\{l_1,\dots,l_k\}$.   In particular,  this proves Corollary \ref{maxelem}.
Applying Lemma \ref{bijection} to $\hat \sigma$ and combining it  with Lemma \ref{limsup} allows one  to find the smallest element of $\hat R$, which is the second-smallest element of $R$; it corresponds to the second largest boundary component. 
Repeating this construction until there remains only a finite number of elements from the original sequence $\sigma$, we can
find all of $R$.  Therefore, the number of boundary components  as well as their lengths are 
uniquely determined by $\sigma$, and the proof of Theorem~\ref{theorem:invspec} is complete. It remains to prove the lemmas.

\bigskip

\emph{Proof of Lemma \ref{limsup}:}
Let $R=\{\alpha_1,\dots,\alpha_k\}$, $0<\alpha_1 \le \alpha_2 \le \dots \le \alpha_k$  be  a multiset such that $F:S(R) \to A$ is a 
close almost-bijection. Without loss of generality, we may assume that $\alpha_1=1$. We first claim that 
\[\limsup_{j}(S(R)_{j+1}-S(R)_j)=\limsup_j(A_{j+1}-A_j).\]
To show this, pick any $\epsilon>0$. Since we are worried only about the limsup, we may ignore any finite number of terms and therefore we assume that $F$ is in fact a bijection. Since $F$ is close, there exists $N\in\mathbb N$ for which $j\geq N$ implies $|F(S(R)_{j})-S(R)_{j}|<\epsilon$. Let $L=\limsup_j(S(R)_{j+1}-S(R)_j)$. Then there are infinitely many $j\geq N$ for which $S(R)_{j+1}-S(R)_j\geq L-\epsilon$ and hence there are no elements of $S(R)$ in the interval $[S(R)_j,S(R)_{j+1}]$ of length at least $L-\epsilon$. But then there are no elements of $A$ in $[S(R)_j+\epsilon,S(R)_{j+1}-\epsilon]$, an interval of length at least $L-3\epsilon$. Since there are infinitely many such $j$, we have $\limsup_j(A_{j+1}-A_j)\geq L-3\epsilon$; since $\epsilon$ was arbitrary, $\limsup_j(A_{j+1}-A_j)\geq\limsup_{j}(S(R)_{j+1}-S(R)_j)$. Interchanging $S(R)$ and $A$ and arguing in the same way we get the opposite inequality,   proving the claim.

We are now ready to prove the lemma. Let $X$ be an 
integer which is a multiple of the numerators of all rational
generators $\alpha_i$, and let $\{\alpha_{i_1},\dots,\alpha_{i_m}\}$, $1< i_1\le \dots \le i_m \le k$,   be the multiset
of irrational generators.  Denote by  $\alpha=\alpha_{i_m}$ the largest of all  irrational generators $\alpha_{i_n}$, $n=1,\dots, m$. By Dirichlet's theorem on simultaneous approximation,
applied to the irrational numbers $X/\alpha_{i_n}$, there is an infinite
set of positive integers $K\subset\mathbb{N}$ such that for each $q\in
K$ and each $\alpha_{i_n}$, there exists an integer $p_{q,n}$ with
\[|\frac{X}{\alpha_{i_n}}-\frac{p_{q,n}}{q}|<\frac{1}{q^{1+1/m}}.\]
Rearranging, we have
\[|qX-p_{q,n}\alpha_{i_n}|<\frac{\alpha_{i_n}}{q^{1/m}}.\]
In particular, for each $q\in K$, there is a multiple of each irrational generator $\alpha_{i_n}$ in the interval $[qX-\frac{\alpha_{i_n}}{q^{1/m}}, qX+\frac{\alpha_{i_n}}{q^{1/m}}]$, and hence in the larger interval $[qX-\alpha/{q^{1/m}},qX+\alpha/{q^{1/m}}]$. Since $qX$ is itself a multiple of each rational generator, there is in fact a multiple of each generator in $[qX-\alpha/{q^{1/m}},qX+\alpha/{q^{1/m}}]$. Since each generator is bigger than 1, there must be no multiples of any
generator, and hence no elements of $S(R)$, in the interval
$[qX-1+\alpha/q^{1/m},qX-\alpha/q^{1/m}]$. Since $K$ is infinite, we conclude that
\[\limsup_{j}(S(R)_{j+1}-S(R)_j)\geq 1.\]
Moreover, since $\alpha_1=1$, each integer is in $S(R)$, so in fact $\limsup_{j}(S(R)_{j+1}-S(R)_j)=1$.
Therefore, by the previous claim, $\limsup_{j}(A_{j+1}-A_j)=1$, which completes the proof. \qed \\

\bigskip

\emph{Proof of Lemma \ref{bijection}:}  Let $N$ be a
large enough number such that for $x\ge N$ in $S(R)$, we have $|F(x)-x|~<~1/10$, and for
$y\ge N$ in $A$, $F^{-1}(y)$ consists of one point.  
First of all, note that $N \ge N_0$, and for $j>N$ the construction implies
$|G_i(j)-j|<1/10$ for $i=1,2$.  Moreover, 
\begin{equation}
\lim_{j\to\infty}|G_i(j)-j|=0\qquad (i=1,2).
\end{equation}

We construct $\hat F$ from $F$, since $S(\hat R)$ is a subset of $S(R)$. Given any $x>N+1$ in $S(\hat R)$, there are two possibilities. One is that $F(x)\in \hat A$, in which case we let $\hat F(x)=F(x)$. The other is that $F(x)\notin \hat A$, in which case $F(x)$ must have been removed from $A$ ``by mistake;" i.e. there is a $j_1$ for which say $F(x)=G_1(j_1)$. Then, obviously, $j_1>N$. In this case $F(j_1)$ cannot have been removed, and we put $\hat F(x):=F(j_1)$. More precisely, there are (at least) two representatives of $j_1$ in the multiset $S(R)$; denote their images under $F$ by $F_1(j_1)$ and $F_2(j_1)$. Suppose that one of these points coincides with $G_1(j_1)$, say $F_2(j_1)=G_2(j_1)$; then we let 
$\hat F(x):=F_1(j_1)$. On the other hand, if neither of these points 
coincide with $G_2(j_1)$, then there exists another point $y\in S(\hat R)$ with $F(y)=G_2(j_1)$. In this case, we let $\hat F(x):=F_1(j_1)$ and $\hat F(y):=F_2(j_1)$. After doing this for each $x>N+1$, we extend this mapping to the finite remainder of the multiset 
$S(\hat R)$ in whichever way we like.

We claim that $\hat F$ is a close almost-bijection. To see this, note that it is well-defined and in fact invertible for large $y$. Therefore it is an almost-bijection. It remains to prove that it is close. Assume that $x>N+1$. Then we either have $\hat F(x)=F(x)$ (and so $|F(x)-x|\to\infty$), or 
$\hat F(x)=F(j_1)$. In the latter case for $i=1,2$ we have 
\[|\hat F(x)-x|=|F(j_1)-G_i(j_1)|\leq |F(j_1)-j_1|+|G_i(j_1)-j_1|,\]
which goes to zero as $j_1$ goes to infinity. This completes the proof. \qed

\begin{remark}
  The results of this section could be used to study multisets that
  are unions of arithmetic progressions. Such multisets were studied,
  for instance, in \cite{Bir}.  In particular, it follows from Lemmas
  \ref{limsup} and \ref{bijection} that there is a close
  almost-bijection between two unions of arithmetic progressions if
  and only if they coincide as multisets.
\end{remark}

\end{document}